\documentclass[reqno]{amsart}
\usepackage{amssymb}
\usepackage{mathrsfs}
\usepackage{amsmath,amstext,amsfonts,verbatim}

\usepackage[latin1]{inputenc}

\usepackage{amsmath,amsthm,amssymb}

\usepackage{amsfonts}

\usepackage{amsmath,amssymb}
\usepackage{graphicx}
\usepackage{color}
\usepackage{indentfirst}

\newtheorem{theorem}{Theorem}[section]
\newtheorem{lemma}[theorem]{Lemma}

\theoremstyle{definition}
\newtheorem{definition}[theorem]{Definition}

\newtheorem{prop}[theorem]{Proposition}
\newtheorem{Corollary}[theorem]{Corollary}
\newtheorem{fact}[theorem]{Fact}

\theoremstyle{remark}

\numberwithin{equation}{section}

\newcommand{\neweq}[1]{\begin{equation}\label{#1}}

\def\phi{\varphi}

\def\incep{\left\{\begin{array}{cl} }
 \def\termin{\end{array}\right. }
\def\2af{2^*_\alpha}

\begin{document}

\title [The systems with almost Banach mean equicontinuity]{\textbf{The systems with almost Banach mean equicontinuity for abelian group actions}}

\author{Bin Zhu}
\address{College of Mathematics and Statistics, Chongqing University, Chongqing 401331, China}
\email{13648338815@163.com}

\author{Xiaojun Huang$^*$}
\address{College of Mathematics and Statistics, Chongqing University, Chongqing 401331, China}
\thanks{$^*$ Corresponding author.}
\email{hxj@cqu.edu.cn}
\thanks{The research was supported by NSF of China (No. 11671057, No. 11471318) and
the Fundamental Research Funds for the Central Universities (No. 2018CDQYST0023).}

\author{Yuan Lian}
\address{College of Mathematics and Statistics, Chongqing University, Chongqing 401331, China}
\email{20140602035@cqu.edu.cn}



\date{}

\begin{abstract}
In this paper, we give the concept of Banach-mean equicontinuity and prove that three concepts, Bnanach-, Weyl- and Besicovitch-mean equicontinuity of a dynamic system with abelian group action are equivalent. Furthermore, we obtain that the topological entropy of a transitive, almost Banach-mean equicontinuous dynamical system with abelain group action is zero. As an application with our main result, we show that the topological entropy of the Banach-mean equicontinuous system under the action of an abelian groups is zero.
\end{abstract}

\maketitle

\section{introduction}

Ergodic theory and topological dynamics are two branches of the modern theory of dynamical systems. The first, thought not in its broadest definition, deals with groups action on a probability measure space in a measure-preserving way; the second deals with the action of groups on compact metric space as groups of homeomorphisms. In this paper, we discuss the problems under the framework of countable group actions on compact metric spaces which constitute fundamental objects of study in the field of dynamical systems.

It is well known that equicontinuous systems have simple dynamical behaviors. A dynamical system is called equicontinuous if the collection of maps defined by the action of the group is a uniformly equicontinuous family. The equicontinuous systems are dynamically the 'simplest' ones; in fact, there is a complete classification of equicontinuous minimal systems.

Mean equicontinuity has received interest in recent years due to connections with ergodic properties of measurable dynamical systems, i.e. dynamical systems equipped with an invariant probability measure. In particular, it has been shown that using a measure theoretic version of mean equicontinuity one can characterize when a measure-preserving system has discrete spectrum \cite{GR} and when the maximal equicontinuous factor is actually an isomorphism \cite{LTY1,BLM}.

Actually, the concept of mean equicontinuity comes in two variants, one known as Weyl-mean equicontinuity and the other as Besicovitch-mean equicontinuity. The concepts of Weyl- and Besi-covitch-mean equicontinuity were introduced in \cite{LTY1} for integer actions. In fact, in this case the notion of Besicovich-mean equicontinuity is immediately seen to be equivalence to the concept of mean Lyapunov-stability which was already introduced in 1951 by Fomin \cite{F} in the context of $\mathbb{Z}$-actions with discrete spectrum. Later, a first systematic treatment was carried out by Auslander \cite{A1}.

Answering an open question in \cite{S}, it was proved by Li, Tu and Ye in \cite{LTY1} that every invariant measure of a mean equicontinuous system of integer group action has discrete spectrum. Localizing the notion of mean equicontinuity, they introduced notions of almost mean equicontinuity and almost Weyl-mean equicontinuity. In \cite{LTY1} they proved that a system with the former property may have positive entropy and meanwhile a system with the latter property must have zero entropy.

Concerning alelian group action, mean equicontinuity and its relation to the spectral theory of dynamical systems (in particular, to discrete spectrum) has been studied by various groups \cite{GR,GRM,FGL}. In the minimal case and the action of abelian group, Fuhrmann, Gr\"{o}ger and Lenz \cite{FGL} concluded that mean equicontinuity is equivalent to discrete spectrum with continuous eigenfunctions.

Inspired by the previous papers, we will discuss the dynamical properties of countable abelian group action systems. In this paper, we introduce the concept of Banach-mean equicontinuity of groups action dynamical systems, which is broader than Weyl- and Besicovitch-mean equicontinuity, and not limited to the dynamical systems of amenable group actions. Moreover, we prove that the above three concepts are equivalent when the dynamic system is an abelian group action. Furthermore, we introduce the concept of almost Banach-mean equicontinuity of countable abelian group action system and obtain the following main result:

\begin{theorem}\label{maintheorem1}
Let $G$ be a countably infinite abelian group, $X$ be a compact metric space and the action $G\curvearrowright X$ be transitive. If the action $G\curvearrowright X$ is almost Banach-mean continuous, then
\begin{equation*}
  h_{{\rm {top}}} (X, G)=0.
\end{equation*}
\end{theorem}

As an application with our main result, we prove that the topological entropy of the Banach-mean equicontinuous system under the action of an abelian groups is zero.

\begin{theorem}
Let $G$ be a countable abelian group, $X$ be a compact metric space and $G\curvearrowright X$ be a continuous action. If $G\curvearrowright X$ is Banach-mean equicontinuous, then
\begin{equation*}
 h_{{\rm {top}}}(X, G)=0
\end{equation*}
\end{theorem}

The paper is organized as follows. We begin in Section 2 by recalling some basic notations, definitions and results of groups action system. In Section 3 we relate the concept and basic propositions of amenable group. Section 4 is devoted to the concepts of Banach- Besicovitch- and Weyl-mean equicontinuity  for amenable group actions. In this section we prove that the above three concepts are equivalent when the dynamic system is an abelian group action. In Section 5 we introduce the concept of Wely-mean sensitivity of amenable group action system. In this section, we obtain a dichotomy result related to Wely-mean equicontinuity and Weyl-mean sensitivity when a dynamical system is transitive. In Section 6, we give the proof of our main resluts.  Finally, in Section 7 we apply our main result to prove the topological entropy of the Banach-mean equicontinuous system under the action of an abelian groups is zero.

\section{Preliminaries}

In this section, we recall some basic notations, definitions, and results.
We refer the reader to the textbook \cite{KL} for information on group action.

By an action of the group $G$ on a set $X$ we mean a map $\alpha: G\times X\rightarrow X$ such that, writing the first argument as a subscript, $\alpha_{s}(\alpha_{t}(x)) = \alpha_{st} (x)$ and $\alpha_{e}(x) =x$ for all $x\in X$ and $s, t\in G$. Most of the time we will write the action as $G\curvearrowright X$ with the image of a pair $(s,x)$ written as $sx$.


\begin{definition}
The action $G\curvearrowright X$ is {\it {(topologically) transitive}} if for all nonempty open sets $U, V\subset X $ there exists an $s\in G$ such that $sU\cap V\not= \emptyset$. The point $x\in X$ is transitive if $\overline{G x}=X$. Denote by Tran$(X,G)$ the set of all transitive points.
\end{definition}

The following proposition in \cite{KL} suggested that, when $X$ is metrizable, transitivity can be thought of as a generic version of minimality in the sense of
Baire category.

\begin{prop}(\cite[\textit{Proposition 7.9}]{KL}) \label{d}
Suppose that $X$ is metrizable. Then the following are equivalent:
\begin{enumerate}
  \item the action $G\curvearrowright X$ is transitive,
  \item there is a dense orbit,
  \item the set of points in $X$ with dense orbit is a dense $G_{\delta}$.
\end{enumerate}
\end{prop}

\begin{definition}
A point $x\in X$ is {\it {recurrent}} if for every neighbourhood $U$ of $x$ the
set $\{s\in G: sx\in U\}$ is infinite. Denote by Re$(X,G)$  the set of all recurrent points.
\end{definition}

\begin{prop}(\cite[\textit{Proposition 7.11}]{KL})\label{e}
Suppose that the action $G\curvearrowright X$ is transitive and that $X$ is
metrizable and has no isolated points. Then the set of recurrent points in $X$ is a
dense $G_{\delta}$.
\end{prop}

\begin{definition}\label{indef01}
Let $X$ be a set. A collection $\{(A_{i, 1}, \ldots, A_{i, k}) : i\in I\}$ of $k$-tuples of subsets of $X$ is said to be {\it {independent}} if
$\bigcap_{i\in F} A_{i, \omega(i)}\not=\emptyset$ for every nonempty finite set $F\subseteq I$ and $\omega\in \{1, \cdots, k\}^F$.
\end{definition}

\begin{definition}\label{indef02}
Let $G\curvearrowright X$ be an action and $\mathbf{A}=(A_{1}, \ldots, A_{k})$ a tuple of subsets of $X$. We say that a set $J\subseteq G$ is an {\it {independence set}} for $\mathbf{A}$ if the collection $\{(s^{-1}A_{1},\ldots , s^{-1}A_{k}): s\in J\}$ is independent.
\end{definition}

\begin{definition}\label{ubd}
Let $G$ be a group. Denote by Fin$(G)$ the family
of all non-empty finite subsets of $G$. Let $E\subseteq G$ be a subset of $G$.
The {\it{upper Banach density}} of $E$ is defined as
\begin{equation*}
 {\rm {BD}}^{*}(E)=\inf\left\{\sup_{g\in G}\frac{|E\cap Fg|}{|Fg|}:F\in {\rm{Fin}}(G)\right\}.
 \end{equation*}
The {\it {lower Banach density}} of $E$ is given by ${\rm {BD}}_{*}(E)=1-{\rm {BD}}^{*}(G\backslash E)$.

Clearly one has ${\rm {BD}}_{*}(E)\leq {\rm {BD}}^{*}(E)$. If ${\rm {BD}}_{*}(E)= {\rm {BD}}^{*}(E)$, then we say that there exists the {\it {Banach density}} of $E$ and denote it by BD$(E)$.
\end{definition}

From the above definition it is easy to see that the Banach upper density has right shift invariant property. For the complete, we give a proof here.

\begin{prop}\label{right}
 ${\rm{BD}}^*(E s)={\rm {BD}}^*(E)$ for any $s\in G$ and $E$ subset of G.
\end{prop}

\begin{proof}
By the symmetry of the pair sets $E s$ and $E$, it is suffice to prove that ${\rm{BD}}^*(E s)\leq{\rm {BD}}^*(E)$. Let $F$ be any nonempty finite subset
of $G$. From the definition of Banach upper density of $E s$, one has
\begin{equation*}
   {\rm{BD}}^*(E s) \leq \sup_{g\in G}\frac{|E s\cap F g|}{|F g|}=\sup_{g\in G}\frac{|E\cap F gs^{-1}|}{|F gs^{-1}|}
    =\sup_{t\in G}\frac{|E\cap F t|}{|F t|}.
\end{equation*}
The arbitrary of $F$ implies that ${\rm{BD}}^*(E s)\leq{\rm {BD}}^*(E)$. Hence the proposition is obtained.
\end{proof}

It is not difficult to observe the following result.

\begin{lemma}\label{lem3}

Let $F,F_{1},F_{2}$ be subsets of $G$ and $s\in G$.
\begin{enumerate}
\item If $F_{1}$ has Banach density one and $F_{1}\subseteq F_{2}$, then so dose $F_{2}$.
\item If $F$ has Banach density one, then $G\backslash F$ is a set of Banach density zero.
\item If $F_{1}$ and $F_{2}$ have Banach density one, then so does $F_{1}\cap F_{2}$.
\item If $F$ has Banach density one, then  so does $Fs$.
\end{enumerate}
\end{lemma}

\section{Amenable group}
This section is devoted to the class of amenable groups. This is a class of groups which plays an important role in
many areas of mathematics such as ergodic theory, harmonic analysis, dynamical systems, geometric group theory, probability theory
and statistics.

Let $G$ be a group. A {\it {mean}} for G on $\ell^{\infty}(G)$ is a unital positive linear functional $\sigma : \ell^{\infty}(G)\rightarrow\mathbb{C}$
(unital means that $\sigma(\mathbf{1})=1$). The mean $\sigma$ is {\it {left invariant}} if $\sigma(sf)=\sigma(f)$ for all $s\in G$ and $f\in\ell^{\infty}(G)$, where $(s f)(t)=f(s^{-1}t)$ for all $t\in G$.

\begin{definition}
The group $G$ is said to be {\it {amenable}} if there is a left invariant mean on $\ell^{\infty}(G)$.
\end{definition}

The above definition of a countable amenable group $G$ is equivalent to the existence of a sequence of finite subsets $\{F_{n}\}$ of $G$ which is
asymptotically invariant, i.e.,
$$
\lim_{n\rightarrow\infty}\frac{|g F_{n}\Delta F_{n}|}{|F_{n}|}=0\text { for all }  g\in G,
$$
where $g F_n=\{g f : f\in F_n\}$, $|\cdot|$ denotes the cardinality of a set, and $\Delta$ is the symmetric difference.
Such sequence is called a {\it {(left) F{\o}lner sequence}}.

The class of amenable groups contains in particular all finite groups, all abelian groups and, more generally, all solvable groups. In this paper, we need the
following theorem.

\begin{theorem}(\cite[\textit{Theorem 4.6.1}]{CC})\label{abiean}
Every abelian group is amenable.
\end{theorem}

\begin{definition}
Let $F$ and $A$ be nonempty finite subsets of $G$. We say that $A$ is
$(F, \varepsilon)$-{\it {invariant}} if $|{s\in A : Fs\subseteq A}|\geq(1-\varepsilon)|A|$.
\end{definition}

\begin{definition}\label{167D}
Let $f$ be a real-valued function on the set of all finite subsets of $G$.
We say that $f(A)$ converges to a limit $L$ as $A$ becomes more and more invariant if for every $\varepsilon> 0$ there are a finite set $F\subseteq G$ and a $\delta> 0$ such that $|f (A) -L| < \varepsilon$ for every nonempty $(F,\delta)$-invariant finite set $A\subseteq G$.
\end{definition}

\begin{theorem}(\cite[\textit{Theorem 4.38}]{KL})\label{hbx01}
Suppose that $G$ is amenable. Let $\phi$ be a $[0,\infty)$-valued function on
the set of all finite subsets of $G$ such that
\begin{enumerate}
\item  $\phi(As) = \phi(A)$ for all finite $A\subseteq G$ and $s\in G$,
\item $\phi(A\cup B)\leq \phi(A) + \phi(B)$ for all finite $A, B\subseteq G$ (subadditivity).
\end{enumerate}
Then $\phi(A)/|A|$ converges to a limit as $A$ becomes more and more invariant.
\end{theorem}

Let $G\curvearrowright X$ be an action and $\mathbf{A}=(A_{1}, \ldots, A_{k})$ a tuple of subsets of $X$.
Suppose that the set $J\subseteq G$ is an independence set for $\mathbf{A}$.
It is readily seen that the function
$$
\varphi_{\mathbf{A}}(F) = \max\{|F\cap J| : J \  \text{is an independence set for} \   \mathbf{A}\}
$$
on the collection of nonempty finite subsets of $G$ satisfies the two conditions in
Theorem \ref{hbx01}, so that the quantity $\varphi_{\mathbf{A}}(F)/|F|$ converges as $F$ becomes more and
more invariant (Definition \ref{167D}) and the limit is equal to inf$_F \varphi_{\mathbf{A}}(F)/|F|$ where $F$
ranges over all nonempty finite subsets of $G$.

\begin{definition}
For a finite tuple $\mathbf{A}=(A_{1}, \ldots, A_{k})$ of subsets of $X$, we define the {\it {independence density}} $I(\mathbf{A})$ of $\mathbf{A}$ to be the above limit.
\end{definition}

\begin{prop}(\cite[\textit{Proposition 12.7}]{KL})\label{c}
Let $\mathbf{A}=(A_{1},\cdots, A_{k})$ be a tuple of subsets of $X$ and let $d>0$.
The following are equivalent:
\begin{enumerate}
\item  $I(\mathbf{A})\geq d$,
\item there are a F{\o}lner sequence $\{F_{n}\}$ and an independence set $J$ for $\mathbf{A}$ such that
\begin{equation*}
  \lim_{n\rightarrow\infty}|F_{n}\cap J|/|F_{n}|\geq d.
\end{equation*}
\end{enumerate}
\end{prop}



In what follows we recall some notions  which were introduced in \cite{P1}.

Let $E\subseteq G$. The {\it {upper asymptotic density}} of $E$ with respect to a F{\o}lner sequence $\mathcal{F}=\{F_{n}\}_{n\in \mathbb{N}}$, denoted by $\overline{d}_{\mathcal{F}}(E)$, is defined by

$$\overline{d}_{\mathcal{F}}(E)=\limsup_{n\rightarrow\infty}\frac{|E\cap F_{n}|}{|F_{n}|}.$$

Similarly, the {\it {lower asymptotic density}} of $E$ with respect to a F{\o}lner sequence $\mathcal{F}=\{F_{n}\}_{n\in \mathbb{N}}$, denoted by $\underline{d}_{\mathcal{F}}(E)$, is defined by $$\underline{d}_{\mathcal{F}}(E)=\liminf_{n\rightarrow\infty}\frac{|E\cap F_{n}|}{|F_{n}|}.$$

One may say $E$ has {\it {asymptotic density}} $d_{\mathcal{F}}(E)$ of $E$ with respect to a F{\o}lner sequence $\mathcal{F}=\{F_{n}\}_{n\in \mathbb{N}}$,  if $\overline{d}_{\mathcal{F}}(E)=\underline{d}_{\mathcal{F}}(E)$, in which $d_{\mathcal{F}}(E)$ is equal to this common value.

Let $\{F_{n}\}_{n\in \mathbb{N}}$ be a F{\o}lner sequence of the amenable group $G$ and $E\subseteq G$. The upper Banach density of $E$ please refer to Definition \ref{ubd}. Meanwhile, we have the following formula for the properties of upper Banach density (see \cite[Lemma 2.9]{DHZ1}).
$$
{\rm {BD}}^{*}(E)=\lim_{n\rightarrow\infty}\sup_{g\in G}\frac{|E\cap F_{n}g|}{|F_{n}g|}.
$$
As for the relationship between upper Banach density and upper asymptotic density, we have the following formula (see \cite[Lemma 3.3]{BBF}).
\begin{equation}\label{ppt}
  {\rm {BD}}^{*}(E)=\sup_{\mathcal{F}}\limsup_{n\rightarrow\infty}\frac{|E\cap F_{n}|}{|F_{n}|},
\end{equation}
where the supremum is taken over all F{\o}lner sequences $\mathcal{F}=\{F_{n}\}_{n\in \mathbb{N}}$ of $G$.

Throughout $G$ is a countable amenable group and $G\curvearrowright X$ is a continuous action on a compact
metric space. We write $\triangle_k(X)$ for the diagonal $\{(x, \cdots, x) : x\in X\}$ in $X^k$.
\begin{definition}\label{b}
We call a tuple $x=(x_{1},\cdots, x_{k})\in X^{k}$ an {\it {IE-tuple}} (or {\it {IE-pair}} if $k=2$) if for every product neighbourhood $U_{1}\times\cdots \times U_{k}$ of $x$ the tuple $(U_{1},\cdots U_{k})$ has positive independence density. We denote the set of IE-tuples of length $k$ by IE$_{k}(X, G)$.
\end{definition}

In this paper, we need the following theorem.

\begin{theorem}(\cite[\textit{Theorem 12.19}]{KL})\label{d02}
${\rm {IE}}_2(X,G)\setminus \Delta_2(X)$ is nonempty if and only if  $h_{{\rm {top}}}(X,G)>0$.
\end{theorem}



\section{Besicovitch-, Weyl- and Banach-mean equicontinuity}

In the study of a dynamical system with bounded complexity (defined by using the
mean metrics), Huang, Li, Thouvenot, Xu and Ye \cite{HLTXY} introduced a notion called equicontinuity in the mean, in 2005.
In 2015, Li, Tu and Ye \cite{LTY1} showed that for a minimal system
the notions of mean equicontinuity and equicontinuity in the mean are equivalent for $\mathbb{Z}$-actions.
The concepts of Besicovitch- and Weyl-mean equicontinuity were introduce in \cite{LTY1} for integer actions and in \cite{FGL} for amenable actions.

In this paper, we give a notion of Banach-mean equicontinuous on a dynamical system for a group action. For countable amenable group action systems, we shows that two concepts Weyl- and Banach-mean equicontinuous are equivalent. By the results of \cite{FGL}, we also know that three concepts of Besicovitch-, Weyl- and Banach-mean equicontinuity are same for abelian group action systems.

\begin{definition}
Let $G$ be a discrete group and Fin$(G)$ be the family
of all non-empty finite subsets of $G$. Let $X$ be a compact metric space with metric $d$. For $x,y\in X$, we denote
$$
\overline{D}(x,y)=\inf_{F\in {\rm {Fin}}(G)}\sup_{g\in G}\frac{1}{|F|}\sum_{t\in Fg}d(tx,ty).
$$
So we call the action $G\curvearrowright X$ is {\it {Banach-mean equicontinuous}} or simply {\it {B-mean equicontinuous}} if for any $\varepsilon>0$, there exists $\delta>0$ such that $\overline{D}(x,y)<\varepsilon$ whenever $d(x,y)<\delta$ for $x, y\in X$.

A point $x\in X$ is called {\it {Banach-mean equicontinuous point}} if for for every $\epsilon>0$, there exists $\delta>0$ such that for every $y\in B(x, \delta)$,
$$
\overline{D}(x,y)<\varepsilon.
$$

We say the action $G\curvearrowright X$ is {\it {almost Banach-mean equicontinuous}} if the dynamic system $(X, G)$ has at least one Banach-mean equicontinuous point.
\end{definition}

By the compactness of $X$, it is easy to see that the action $G\curvearrowright X$ is Banach-mean equicontinuous if and only if every point in $X$ is Banach-mean equicontinuous point.

In \cite{FGL}, Fuhrmann, Gr\"{o}ger and Lenz introduced the concepts Besicovitch- and Weyl-mean equicontinuity for amenable action systems.

\begin{definition}\label{mmm}
Let $G$ be an amenable group and $\mathcal{F}=\{F_{n}\}_{n\in \mathbb{N}}$ be a F{\o}lner sequence of $G$. We call the action $G\curvearrowright X$ is  {\it {Besicovitch-$\mathcal{F}$-mean equicontinuous}} if for every $\varepsilon>0$, there exists $\delta>0$ such that
$$
D_{\mathcal{F}}(x, y):=\limsup_{n\rightarrow\infty}\frac{1}{|F_{n}|}\sum_{g\in F_{n}}d(gx,gy)<\varepsilon,
$$
for all $x, y\in X$ with $d(x, y)<\delta$. The dependence on the F{\o}lner sequence immediately motivates the next definition. We say the action $G\curvearrowright X$ is  {\it {Weyl-mean equicontinuous}} if for every $\varepsilon>0$, there exists $\delta>0$ such that for all $x, y\in X$ with $d(x, y)<\delta$ we have
\begin{equation}\label{weyl}
  D(x, y):=\sup \ \{D_{\mathcal{F}}(x, y) : \mathcal{F} \  \text {  is a F{\o}lner sequence}\}<\varepsilon.
\end{equation}

A point $x\in X$ is called {\it {Weyl-mean equicontinuous point}} if for for every $\epsilon>0$, there exists $\delta>0$ such that for every $y\in B(x, \delta)$,
$$
D(x, y)<\varepsilon.
$$

We say the action $G\curvearrowright X$ is {\it {almost Weyl-mean equicontinuous}} if the dynamic system $(X, G)$ has at least one Weyl-mean equicontinuous point.
\end{definition}

Before we can proceed, a few comments are in order. First, note that $\overline{D}$, $D_{\mathcal{F}}$ and $D$ are pseudometric. Moreover, as is not hard to see, $D$ is $G$-invariant, that is, $D(gx, gy)=D(x, y)$ for all $x, y\in X$ and $g\in G$.

In the following, for the amenable group action system, we will see that the Banach pseudometric $\overline{D}(\cdot, \cdot)$ equals to the Wely pseudometric $D(\cdot, \cdot)$.

\begin{theorem}
Let $G$ be a countable amenable group, $X$ be a compact metric space and $G\curvearrowright X$ be a dynamic system. Then
\begin{equation*}
  \overline{D}(x, y)=D(x, y) \quad {\text {for any pair  }}\  x, y\in X.
\end{equation*}
\end{theorem}

\begin{proof}
Let $x, y\in X$.
Firstly, we show that $D(x, y)\leq \overline{D}(x, y)$.

Let $\varepsilon>0$.
From the definition of $\overline{D}(x, y)$,
there is a nonempty finite subset $F\in {\rm {Fin}}(G)$ such that
$$
\sup_{g\in G}\frac{1}{|F|}\sum_{t\in Fg}d(tx,ty)<\overline{D}(x, y)+\varepsilon.
$$

Let  $\{F_{n}\}_{n\in \mathbb{N}}$ be a F{\o}lner sequence of $G$. In the following we will show that
$$
\limsup_{n\rightarrow\infty}\frac{1}{|F_{n}|}\sum_{t\in F_{n}}d(tx,ty)\leq \overline{D}(x, y)+\varepsilon.
$$

Given $g\in G$. For every $h\in F_{n}$, one has
$$
\frac{1}{|F|}\sum_{s\in Fhg}d(sx,sy)\leq\sup_{g'\in G}\frac{1}{|F|}\sum_{s\in Fg'}d(sx,sy)<\overline{D}(x, y)+\varepsilon.
$$
Thus it follows that
$$
\sum_{h\in F_{n}}\sum_{s\in Fhg}d(sx,sy)<|F_{n}||F|(\overline{D}(x, y)+\varepsilon).
$$
We denote by $\alpha(h,t)=d((thg)x, (thg)y)$ for $h\in F_{n}$ and $t\in F$. Then the above inequality can be re-written as
$$
\sum_{h\in F_{n}}\sum_{t\in F}\alpha(h, t)<|F_{n}||F|(\overline{D}(x, y)+\varepsilon).
$$
It is clear that there is $t'\in F$ such that
$$
\sum_{h\in F_{n}}\alpha(h,t')=\min\left\{\sum_{h\in F_{n}}\alpha(h,t): t\in F\right\}.
$$
Therefore, we get
$$
(\overline{D}(x, y)+\varepsilon)|F_{n}||F|>\sum_{h\in F_{n}}\sum_{t\in F}\alpha(h,t)=\sum_{t\in F}\sum_{h\in F_{n}}\alpha(h,t)\geq |F| \sum_{h\in F_{n}}\alpha(h,t'),
$$
which implies
$$
\sum_{s\in t' F_{n} g} d(s x, s y)=\sum_{h\in F_{n}}\alpha(h,t')<|F_{n}|(\overline{D}(x, y)+\varepsilon).
$$
Note that
\begin{equation*}
  \begin{split}
  \sum_{s\in F_{n} g} d(s x, s y) & \leq \sum_{s\in t' F_{n} g} d(s x, s y)
+\sum_{s\in t' F_{n} g\triangle F_{n} g} d(s x, s y) \\
    &<|F_n|(\overline{D}(x, y)+\varepsilon)
+|t' F_{n} g\triangle F_{n} g|\cdot {\rm {diam}}(X)\\
& =|F_n|(\overline{D}(x, y)+\varepsilon)+|t' F_{n}\triangle F_{n}|\cdot {\rm {diam}}(X)
\end{split}
\end{equation*}
where ${\rm {diam}}(X)$ is the diameter of the compact metric space $(X, d)$.
Since $F_n$ is a F{\o}lner sequence, we have
\begin{equation*}
  \begin{split}
 \limsup_{n\rightarrow\infty} \frac{1}{|F_{n}|}\sum_{s\in F_{n}g}d(sx, sy) & \leq \overline{D}(x, y)+\varepsilon+\limsup_{n\rightarrow\infty} \frac{|t' F_{n}\triangle F_{n}|}{|F_n|} {\rm {diam}}(X)\\
      & = \overline{D}(x, y)+\varepsilon.
  \end{split}
\end{equation*}
From the arbitrariness of the F{\o}lner sequence $\{F_n\}$ we get
$$
\sup_{\{F_{n}\}}\  \limsup_{n\rightarrow\infty}\frac{1}{|F_{n}|}\sum_{g\in F_{n}}d(gx,gy)\leq \overline{D}(x, y)+\varepsilon
$$
where the supremum is taken over all F{\o}lner sequence of $G$.
That is
\begin{equation*}
  D(x, y)\leq \overline{D}(x, y)+\varepsilon.
\end{equation*}
The arbitrariness of $\varepsilon$ implies that
\begin{equation*}
  D(x, y)\leq \overline{D}(x, y).
\end{equation*}

Suppose that $D(x_0, y_0)<\overline{D}(x_0, y_0)$ for some $x_0, y_0\in X$. In what follows we will obtain a contradiction.

We choose two real number $\eta_1, \eta_2\in\mathbb{R}$ such that
\begin{equation}\label{MP006}
  D(x_0, y_0)<\eta_1<\eta_2<\overline{D}(x_0, y_0).
\end{equation}

Let $\{F_{n}\}_{n\in \mathbb{N}}$ be a F{\o}lner sequence of the amenable group $G$. Note that $F_{n}$ is a nonempty finite subset of $G$ for each $n\in\mathbb{N}$. From the definition of $\overline{D}(x_{0}, y_{0})$, we have
$$
\sup_{g\in G}\frac{1}{|F_{n}|}\sum_{t\in F_{n}g}d(tx_{0},ty_{0})\geq \overline{D}(x_{0}, y_{0})>\eta_2.
$$
Thus, for each $n\in\mathbb{N}$, there exists $g_{n}\in G$ such that
$$
\frac{1}{|F_{n}g_{n}|}\sum_{t\in F_{n}g_{n}}d(tx_{0},ty_{0})>\eta_2.
$$
Set $H_{n}=F_{n}g_{n}$. Since $\mathcal{F}':=\{H_{n}\}_{n\in \mathbb{N}}$ is also a (left) F{\o}lner sequence of $G$, we get
$$
\eta_1>D(x_0, y_0)\geq D_{\mathcal{F}'}(x_0, y_0)=\limsup_{n\rightarrow\infty}\frac{1}{|H_{n}|}\sum_{g\in H_{n}}d(gx_{0},gy_{0})\geq\eta_2.
$$
This is a contradiction. Hence the theorem is proved.
\end{proof}

From above Theorem it follows that two concepts of Banach- and Weyl-mean equicontinuous are equivalent for the amenable group action system.

\begin{Corollary}
Let $G$ be a countable amenable group, $X$ be a compact metric space and $G\curvearrowright X$ be a dynamic system. Then $G\curvearrowright X$ is Banach-mean equicontinuous if and only if
$G\curvearrowright X$ is Weyl-mean equicontinuous.
\end{Corollary}

According to the theorem on the independence of F{\o}lner sequences of an amenable group in \cite[\textit{Theorem 1.3, p.\,6}]{FGL}, we can get the following result.

\begin{theorem}
Let $G$ be a countable abelian group and $(X, G)$ be a dynamical system. Then the following three statements are equivalent.
\begin{enumerate}
  \item $(X, G)$ is Banach-mean equicontinuous.
  \item $(X, G)$ is Weyl-mean equicontinuous.
  \item $(X, G)$ is Besicovitch-$\mathcal{F}$-mean equicontinuous for some left F{\o}lner sequence $\mathcal{F}$.
\end{enumerate}
\end{theorem}

Let $G$ be a countable amenable group and $(X, G)$ be a dynamical system.
Let $\mathcal{E}$ denote the set of all Weyl-mean equicontinuous points of the dynamic system $(X, G)$. For every $\varepsilon>0$, let
\begin{equation}\label{xxx}
 \mathcal{E}_{\varepsilon}=\bigg\{x\in X : \exists \, \delta>0,\, \forall\, y, z\in B(x,\delta), \, D(y, z)<\varepsilon \bigg\}.
\end{equation}

For the Weyl-mean equicontinuous points we have the following proposition.

\begin{prop}\label{g}
Let $G$ be a countable amenable group, $(X, G)$ be a dynamical system and $\varepsilon>0$. Then
$\mathcal{E}_{\varepsilon}$ is open and $s \,\mathcal{E}_{\varepsilon/2}\subseteq \mathcal{E}_{\varepsilon}$ for all $s\in G$. Moreover, $\mathcal{E}=\bigcap\limits_{m=1}^{\infty}\mathcal{E}_{\frac{1}{m}}$ is a $G_{\delta}$ subset of $X$.
\end{prop}
\begin{proof}
Let $x\in \mathcal{E}_{\varepsilon}$. Choose $\delta>0$ satisfying the condition from the definition of $\mathcal{E}_{\varepsilon}$ for $x$. Fix $y\in B(x, \delta/2)$. If $z,w\in B(y, \delta/2)$, then $z,w\in B(x,\delta)$, so $D(z, w)<\varepsilon$.
This shows that $B(x, \delta/2)\subseteq \mathcal{E}_{\varepsilon}$ and hence $\mathcal{E}_{\varepsilon}$ is open.

Let $s\in G$. Suppose $x\in s \,\mathcal{E}_{\varepsilon/2}$, thus $s^{-1}x\in \mathcal{E}_{\varepsilon/2}$. Choose $\delta>0$ satisfying the condition from the definition of $\mathcal{E}_{\varepsilon/2}$ for $s^{-1} x$. That is, for all $y,z\in B(s^{-1}x, \delta)$, one has $D(y, z)<\varepsilon/2.$
By the continuity of the map $s^{-1}$, there exists $\eta>0$ such that $d(s^{-1}y,s^{-1}x)<\delta$ for any $y\in B(x,\eta)$.

Let $u,v\in B(x,\eta)$, then $s^{-1}u,s^{-1}v\in B(s^{-1}x, \delta)$.

Let $\mathcal{F}=\{F_n\}_{n\in\mathbb{N}}$ be any F{\o}lner sequence of $G$. Note that $\mathcal{F} s=\{F_n s\}_{n\in\mathbb{N}}$ is also a (left) F{\o}lner sequence of $G$.
Thus, we have
\begin{equation*}
\begin{split}
D_{\mathcal{F}}(u, v)&=\limsup_{n\rightarrow\infty}\frac{1}{|F_{n}|}\sum_{g\in F_{n}} d(g s s^{-1} u, g s s^{-1} v)\\
&=\limsup_{n\rightarrow\infty}\frac{1}{|F_{n} s|}\sum_{t\in F_{n} s}d(t (s^{-1}u), t (s^{-1}v))\\
&=D_{\mathcal{F}s}(s^{-1}u, s^{-1}v)\leq D(s^{-1}u, s^{-1}v)<\varepsilon/2.
\end{split}
\end{equation*}
The arbitrariness of the F{\o}lner sequence $\mathcal{F}$ indicates that $D(u, v)\leq\varepsilon/2<\varepsilon$ which implies $x\in \mathcal{E}_{\varepsilon}$. Hence we get $s \,\mathcal{E}_{\varepsilon/2}\subseteq \mathcal{E}_{\varepsilon}$.

If $x\in X$ belongs to all $\mathcal{E}_{\frac{1}{m}}$, then clearly $x\in \mathcal{E}$.

Conversely, if $x\in \mathcal{E}$ and $m\geq1$, there exists $\delta>0$ such that $D(x, y)<1/2m$ for all $y\in B(x,\delta)$. If $y,z\in B(x,\delta)$, then
\begin{equation*}
D(y, z)\leq D(y, x)+D(x, z)<\frac{1}{m}.
\end{equation*}
Thus $x\in \mathcal{E}_{\frac{1}{m}}$. Therefore we get $\mathcal{E}=\bigcap\limits_{m=1}^{\infty}\mathcal{E}_{\frac{1}{m}}$. Hence the proof is completed.
\end{proof}

\section{Weyl-mean sensitivity}
Let $X$ be a compact metric space.
Recall that a subset set of $X$ is called {\it {residual}} if it contains the intersection of a countable collection of dense open sets. By the Baire category theorem, a residual set is also dense in $X$.

\begin{definition}
Let $G\curvearrowright X$ be a continuous action and $x\in X$ be a point.
We call the point $x$ is  {\it {Weyl-mean sensitive point}} if there exists $\delta>0$ such that for every $\varepsilon>0$, there is $y\in B(x,\varepsilon)$ satisfying
$$
D(x, y)>\delta.
$$
Here the definition of the function $D(\cdot\,,\cdot)$ please refer to (\ref{weyl}).

We call the action $G\curvearrowright X$ is {\it {Weyl-mean sensitive}} if every point $x\in X$ is the Weyl-mean sensitive point.
\end{definition}

\begin{prop}
Let $G$ be a countable amenable group and $X$ a compact metric space.
Let $G\curvearrowright X$ be a transitive system.
\begin{enumerate}
\item  The set of Weyl-mean equicontinuous points is either empty or residual. If in addition the action $G\curvearrowright X$ is almost Weyl-mean equicontinuous, then every transitive point is Weyl-mean equicontinuous.
\item If the action $G\curvearrowright X$ is minimal and almost Weyl-mean equicontinuous, then it is Weyl-mean equicontinuous.
\end{enumerate}
\end{prop}

\begin{proof}
If $\mathcal{E}_{\varepsilon}$ is empty for some $\varepsilon>0$ then the set of Weyl-mean equicontinuous points $\mathcal{E}$ is empty.

Now, we assume that
every $\mathcal{E}_{\varepsilon}$ is nonempty. Then, for each $\varepsilon>0$, $\mathcal{E}_{\varepsilon}$ is nonempty open subset of $X$.
In what follows we show that every $\mathcal{E}_{\varepsilon}$ is dense. Let $U$ be any nonempty open subset of $X$.
By the transitivity of the action $G\curvearrowright X$, noting $\mathcal{E}_{\varepsilon/2}$ being nonempty open subset of $X$ and Proposition \ref{g}, there exists $s\in G$ such that $\emptyset\neq U\cap s\,\mathcal{E}_{\varepsilon/2}\subseteq U\cap \mathcal{E}_{\varepsilon}$.

Hence $\mathcal{E}$ is either empty or residual by the Baire Category Theorem.

If $\mathcal{E}$ is residual, then every $\mathcal{E}_{\varepsilon}$ is open and dense. Let $x\in X$ is a transitive point and $\varepsilon>0$. Then there exists some element $s\in G$ such that $s\,x\in \mathcal{E}_{\varepsilon/2}$, and by Proposition \ref{g}, $x\in s^{-1}\mathcal{E}_{\varepsilon/2} \subseteq\mathcal{E}_{\varepsilon}$. Thus $x\in \mathcal{E}$.

\end{proof}

\begin{prop}\label{Prop6}
Let $G$ be a countable amenable group and $X$ a compact metric space.
Let $G\curvearrowright X$ be a continuous action. If there exists $\delta>0$ such that for every non-empty open subset $U$ of $X$, there are $x,y\in U$ satisfying $D(x, y)>\delta$.
Then the dynamical system $(X, G)$ is Weyl-mean sensitivity.
\end{prop}

\begin{proof}
Suppose that there exists $\delta>0$ such that for any nonempty open subset $U$ of $X$, there are $u, v\in U$ satisfying
$
D(u, v)>2\delta.
$

Let $x\in X$ and $\varepsilon>0$, then $B(x,\varepsilon)\neq\emptyset$ and $B(x,\varepsilon)$ is open subset of $X$. Then there exist $y,z\in B(x,\varepsilon)\subseteq X$ satisfying
$
D(y, z)>2\delta.
$

Thus we have either
$
D(x, y)>\delta
$
or
$
D(x, z)>\delta.
$
which implies that the dynamical system $(X, G)$ is Weyl-mean sensitivity.
\end{proof}

\begin{prop}\label{Prop7}
Let $G$ be a countable amenable group and $X$ a compact metric space.
Let the action $G\curvearrowright X$ be transitive. If there exists a transitive point which is a Weyl-mean sensitivity point then the action $G\curvearrowright X$ is a Weyl-mean sensitivity.
\end{prop}

\begin{proof}
Let $x\in X$ be a Weyl-mean sensitivity point. Thus there exists $\delta>0$ such that, for every $\varepsilon>0$, there is $y\in B(x,\varepsilon)$ satisfying
$
D(x, y)>\delta.
$

Given a nonempty open subset $U$ of $X$. Since $x$ is a transitive point, there exists $s\in G$ such that $s x\in U$, that is, $x\in s^{-1} U$. Furthermore, as $s^{-1} U$ is open, there is $\epsilon>0$ such that $B(x,\epsilon)\subseteq s^{-1}U$, that is, $sB(x,\epsilon)\subseteq U$. By the assumption that $x$ a Weyl-mean sensitivity point, then there exists $y\in B(x,\epsilon)$ satisfying
$
D(x, y)>\delta.
$
By the definition of $D(x, y)$, there is a (left) F{\o}lner sequence $\mathcal{F}=\{F_n\}_{n\in\mathbb{N}}$ of $G$ such that
\begin{equation*}
 D_{\mathcal{F}}(x, y)>\delta.
\end{equation*}

Let $u=s x$, $v=s y$. Noting that $\mathcal{F}s=\{F_n s^{-1}\}_{n\in\mathbb{N}}$ being also a (left) F{\o}lner sequence and $u, v\in U$, then
\begin{equation*}
D(u, v)\geq D_{\mathcal{F}s^{-1}}(u, v)=D_{\mathcal{F}s^{-1}}(s x, s y)
=D_{\mathcal{F}}(x, y)>\delta.
\end{equation*}
Therefore the action $G\curvearrowright X$ is Weyl-mean sensitivity by Proposition \ref{Prop6}.
\end{proof}

\begin{theorem}\label{relation}
Let $G$ be a countable amenable group and $X$ a compact metric space. If the action $G\curvearrowright X$ is transitive, then the action $G\curvearrowright X$ is either almost Weyl-mean equicontinuous  or Weyl-mean sensitivity.
\end{theorem}

\begin{proof}
Let $x\in X$ be a transitive point. If $x$ is a Weyl-mean sensitivity point, then the action $G \curvearrowright X$ is Weyl-mean sensitivity by Proposition \ref{Prop7}. If $x$ is not a Weyl-mean sensitivity point, then it is a Weyl-mean equicontinuous point. So the action $G\curvearrowright X$ is  almost Weyl-mean equicontinuous.
\end{proof}

\begin{Corollary}
Let $G$ be a countable amenable group and $X$ a compact metric space. Let $G\curvearrowright X$ be a minimal system. Then the action $G\curvearrowright X$ is either Weyl-mean sensitivity or Weyl-mean equicontinuous.
\end{Corollary}

\section{The proof of Main Theorem}

To prove our main theorem we need some preparation. The following result please see \cite[Proposition 5.8]{LTY1}.

\begin{prop}\label{me.}
Let $(X,\beta,\mu)$ be a probability space, and $\{E_{i}\}_{i=1}^\infty$ be a sequence of measurable sets with $\mu(E_i)\ge a>0$ for some constant $a$ and any $i\in \mathbb{N}$. Then for any $k\ge1$ and $\epsilon>0$, there is $N=N(a,k,\epsilon)$ such that for any tuple $\{s_1<s_2<...<s_n\}$ with $n\ge N$, there exist $1\le t_1<...<t_k\le n$ with
$$\mu(E{_{s{_{t{_1}}}}}\cap E{_{s{_{t{_2}}}}}\cap...\cap E{_{s{_{t{_k}}}}})\ge a^k-\epsilon.$$
\end{prop}

Let $(X, d)$ be a compact metric space, with Borel $\sigma$-algebra $\mathscr{B}$. Denote by $\mathcal{M}(X)$ the space of Borel probability measures on $X$. Our main interest is its weak-$*$ topology of space $\mathcal{M}(X)$. It is standard, a convenient source is Pathasarathy \cite{P}.

\begin{theorem}\label{MP}
Let $X$ be compact metric space and
$\{\mu_n\}$ be a sequence of probability measures in $\mathcal{M}(X)$. Let $\mu\in\mathcal{M}(X)$. Then the following statements are equivalent:
\begin{enumerate}
  \item $\{\mu_n\}$ converges to $\mu$ with weak-$*$ topology in $\mathcal{M}(X)$;
  \item $\lim_{n\rightarrow\infty}\int f d \mu_n=\int f d\mu$ for all $f\in C(X)$;
  \item $\limsup_{n\rightarrow\infty} \mu_n(C) \leq \mu(C)$ for every closed set $C$;
  \item $\liminf_{n\rightarrow\infty} \mu_n(G) \geq \mu(G)$ for every open set $G$;
  \item $\lim_{n\rightarrow\infty} \mu_n(A)=\mu(A)$ for every Borel set $A$ whose boundary has $\mu$-measure 0.
\end{enumerate}
\end{theorem}

In order to obtain our results, we need the following elemental fact.

\begin{fact}\label{MP003}
Let $\{a_n\}$ and $\{b_n\}$ be two sequences of real numbers. Suppose that $\lim_{n\rightarrow\infty} b_n$ exists and $\limsup_{n\rightarrow\infty} a_n$ is finite. Then
\begin{equation*}
  \limsup_{n\rightarrow\infty}\  (a_n+b_n)=\limsup_{n\rightarrow\infty} a_n+\lim_{n\rightarrow\infty} b_n.
\end{equation*}
\end{fact}

The following result is the well-known Furstenberg corresponding principle \cite{F1} of the amenable group version.
\begin{prop}\label{mea.}
Let $G$ be a countable amenable group and
$F$ be a subset of $G$ with BD$^*(F)>0$. Then for any $k\ge 1$ and $\epsilon>0$, there is $N=N({\rm {BD}}^*(F), k, \epsilon)$ such that for any $n\geq N$ and any tuple $\{s_1, s_2, \cdots , s_n\}\subseteq G$, there exist $\{1\leq t_1<t_2<\cdots<t_k\leq n\}$ such that
$$ {\rm {BD}}^*({s_{t_1}}^{-1}F\cap {s_{t_2}}^{-1}F\cap\cdots\cap {s_{t_k}}^{-1}F)\ge({\rm {BD}}^*(F))^k-\epsilon.$$
\end{prop}

\begin{proof}
Let $K=\{0, 1\}$ be a finite alphabet. We define the map  $\Sigma: G\times K^G \rightarrow K^G$ by $\Sigma(g,x):=x \circ {R_g}$, for all $g\in $G and $x\in K^G$. Here $R_g: G \rightarrow G$ is defined by $R_g(h):=hg$ for all $h \in G$.
Let $s\in G$ and $x\in K^G$. Thus $s x(g)=x(g s)$ for all $g\in G$.

Given $\xi\in K^G $ satisfying $\xi(s)=0$ for all $s\in F$ and $\xi(s)=1$ for all $s\in G\backslash F$.
Denote by
\begin{equation*}
 X=\overline{\{s\xi: s\in G\}} \quad {\text {the orbit closure with respect to }}\  K^G.
\end{equation*}
It is clear that $X$ is a compact metric space. Meanwhile,
it is follows that the set $\{s\in G: \xi(s)=0\}=F$ has positive upper Banach density.

Let $\{H_{n}\}$ be a (left) F{\o}lner sequence for $G$.
By a formula of upper Banach density (see \cite[Lemma 2.9]{DHZ1}), we have
$$\lim_{n\rightarrow\infty}\sup_{g\in G}\frac{\mid F\cap H_{n}g\mid}{\mid H_{n}g\mid}={\rm {BD}}^*(F)>0.$$

Let $\kappa>0$. For the above limit equation, there is $N_1\in\mathbb{N}$ such that, for every $n\geq N_1$, one has
$${\rm {BD}}^*(F)-\frac{\kappa}{2}<\sup_{g\in G}\frac{\mid F\cap H_{n}g\mid}{\mid H_{n}g\mid}<{\rm {BD}}^*(F)+\frac{\kappa}{2}.$$
Thus, for each $n\geq N_1$, there is $g_{n}\in G$ such that
$$
{\rm {BD}}^*(F)-\kappa<\sup_{g\in G}\frac{\mid F\cap H_{n}g\mid}{\mid H_{n}g\mid}-\frac{\kappa}{2}<\frac{\mid F\cap H_{n}g_{n}\mid}{\mid H_{n}g_{n}\mid}\leq\sup_{g\in G}\frac{\mid F\cap H_{n}g\mid}{\mid H_{n}g\mid}+\frac{\kappa}{2}<{\rm {BD}}^*(F)+\kappa.
$$
That means
$$
\lim_{n\rightarrow\infty}\frac{\mid F\cap H_{n}g_{n}\mid}{\mid H_{n}g_{n}\mid}={\rm {BD}}^*(F)>0.
$$

Set $F_{n}=H_{n}g_{n}$ for all $n\in\mathbb{N}$. So $\{F_{n}\}_{n\in\mathbb{N}}$ is a (left) F{\o}lner sequence with
\begin{equation}\label{MP004}
  \lim_{n \rightarrow \infty}\frac{|F_{n}\cap F|}{|F_{n}|}={\rm {BD}}^*(F).
\end{equation}

Define a sequence of probability measures in $\mathcal{M}(X)$ as follows:
\begin{equation*}
  \mu_n:=\frac{1}{|F_n|}\sum_{s\in F_n} \delta_{s\, \xi}
\end{equation*}
where $\delta_{s\,\xi}$ is Dirac measure at the point $s\,\xi$ in $X$.

Since $\mathcal{M}(X)$ is compact metrizable space(Theorem 6.3 in \cite{P}, p43), there exists a subsequence $\{\mu_{n_l}\}_{l\in\mathbb{N}}$ converges to a probability measure $\nu$ with weak-$*$ topology in $\mathcal{M}(X)$, that is,
\begin{equation*}
  \mu_{n_l}=\frac{1}{|F_{n_l}|}\sum_{s\in F_{n_l}} \delta_{s\, \xi}\xrightarrow{{\text {weak-}}*} \nu.
\end{equation*}

In what follows we will show that $\nu$ is a $G$-invariant probability measure, that is, $\nu=g\nu$ for each $g\in G$.

Let $g\in G$. From (3) of Theorem \ref{MP}, it is easy to check that
\begin{equation}\label{MP002}
  g \mu_{n_l}\xrightarrow{{\text {weak-}}*} g\nu.
\end{equation}

\noindent {\bf {Claim.}}\  For any Borel set $B$ of $X$, one has
\begin{equation}\label{MP001}
  \lim_{l\rightarrow\infty} \big(g\mu_{n_l}(B)-\mu_{n_l}(B)\big)=0.
\end{equation}
In fact, since $\{F_{n_l}\}_{l\in\mathbb{N}}$ is a F{\o}lner sequence of $G$, we have
\begin{equation*}
  \begin{split}
   |g\mu_{n_l}(B)-\mu_{n_l}(B)|  & =|\mu_{n_l}(g^{-1}B)-\mu_{n_l}(B)|
       =\frac{1}{|F_{n_l}|}\left(\sum_{s\in F_{n_l}}\delta_{s\, \xi}(g^{-1}B)-\sum_{s\in F_{n_l}}\delta_{s\, \xi}(B)\right)\\
      & =\frac{1}{|F_{n_l}|}\left(\sum_{s\in F_{n_l}}\delta_{(gs)\, \xi}(B)-\sum_{s\in F_{n_l}}\delta_{s\, \xi}(B)\right)\\
      & =\frac{1}{|F_{n_l}|}\left(\sum_{s\in g F_{n_l}}\delta_{s\, \xi}(B)-\sum_{s\in F_{n_l}}\delta_{s\, \xi}(B)\right)\\
      & \leq \frac{|g F_{n_l}\Delta F_{n_l}|}{|F_{n_l}|}\rightarrow 0\quad {\text {as  }}\   l\rightarrow\infty.
  \end{split}
\end{equation*}
Hence the claim is obtained.

Now we will prove $g\nu=\nu$. Let $C$ be any closed subset of $X$. From (\ref{MP001}), Fact \ref{MP003} and (3) of Theorem \ref{MP}, we have
\begin{equation*}
  \begin{split}
  \limsup_{l\rightarrow\infty} g\mu_{n_l}(C)&=\limsup_{l\rightarrow\infty} g\mu_{n_l}(C)+\lim_{l\rightarrow\infty} \big(\mu_{n_l}(C)-g\mu_{n_l}(C)\big) \\
      & =\limsup_{l\rightarrow\infty} \mu_{n_l}(C)\leq \nu(C).
  \end{split}
\end{equation*}
Applying (3) of Theorem \ref{MP} again, we get
\begin{equation*}
  g \mu_{n_l}\xrightarrow{{\text {weak-}}*} \nu.
\end{equation*}
Combining with (\ref{MP002}), we have $g\nu=\nu$. So $\nu$ is a $G$-invariant measure.

Denote by $e$ the unit element of the group $G$.
We define $A(0)=\{\eta\in K^G : \eta(e)=0\}\cap X$. Since $\{\eta\in K^G : \eta(e)=0\}$ is a clopen subset of $K^G$, it follows that
$A(0)$ is a clopen subset of $X$. So the boundary of $A(0)$ is empty set, that is, $\partial (A(0))=\emptyset$. From (5) of Theorem \ref{MP},  we have
$$
\nu(A(0))=\lim_{l \rightarrow \infty}\mu_{n_l}(A(0))=\lim_{l \rightarrow \infty}\frac{1}{|F_{n_l} | }\sum_{s\in F_{n_l}}\delta_{s\, \xi}(A(0)).
$$
Note that $s \,\xi\in A(0)\Leftrightarrow\xi(s)=0\Leftrightarrow s\in F$. So, by (\ref{MP004}), we get
\begin{equation}\label{MP005}
 \nu (A(0))=\lim_{l \rightarrow \infty}\frac{1}{|F_{n_l} | }\sum_{s\in F_{n_l}}\delta_{s\, \xi}(A(0))=\lim_{l \rightarrow \infty}\frac{|F\cap F_{n_l}|}{|F_{n_l} | }={\rm {BD}}^*(F).
\end{equation}

Since $G$ is countable, we list $G$ as $G=\{g_i\}_{i=1}^{\infty}$. Denote by $E_i=g_i^{-1} A(0)$ for each $i\in \mathbb{N}$. Owing to $\nu$ being
$G$-invariant, we deduce that $\nu(E_i)=\nu(g_i^{-1} A(0))=\nu(A(0))={\rm {BD}}^*(F)$ for each $i\in \mathbb{N}$.

Let $k\geq 1$ and $\epsilon>0$. From Proposition \ref{me.}, there is $N=N({\rm {BD}}^*(F),k,\epsilon)$ such that, for any $n\geq N$ and any tuple $\{s_1, s_2, \cdots , s_n\}\subseteq G$, there exist $\{1\leq t_1<t_2<\cdots <t_k\leq n\}$ satisfying
\begin{equation}\label{ktime}
\nu(s{_{t{_1}}}^{-1}A(0)\cap s{_{t{_2}}}^{-1}A(0)\cap...\cap s{_{t{_k}}}^{-1}A(0))\ge \nu(A(0))^k-\epsilon.
\end{equation}
Set $B=s{_{t{_1}}}^{-1}A(0)\cap s{_{t{_2}}}^{-1}A(0)\cap...\cap s{_{t{_k}}}^{-1}A(0)$. It is easy to see that
$$
s\xi\in s_{t_i}^{-1}A(0)\Leftrightarrow (s_{t_i}s)\xi\in A(0)\Leftrightarrow (s_{t_i}s)\xi(e)=0\Leftrightarrow \xi(s_{t_i}s)=0 \Leftrightarrow s\in {s_{t_i}}^{-1}F.
$$
Since $A(0)$ is a clopen subset of $X$, the set $B$ is also clopen in $X$. Therefore, the boundary of $B$ is empty set, that is, $\partial(B)=\emptyset$.
Thus, by the equation (\ref{ppt}) and (5) of Theorem \ref{MP}, we have
\begin{equation}\label{Bade}
\begin{split}
&{\rm {BD}}^*\left({s^{-1}_{t_1}}F\cap {s^{-1}_{t_2}}F\cap...\cap {s^{-1}_{t_k}}F\right)\\
&\geq\limsup_{l \rightarrow \infty}\frac{|F_{n_l}\cap({s^{-1}_{t_1}}F\cap {s^{-1}_{t_2}}F\cap...\cap {s^{-1}_{t_k}}F)|}{|F_{n_l}|}\\
&=\limsup_{l \rightarrow \infty}\frac{1}{|F_{n_l}|}\sum_{s\in F_{n_l}}\delta_{s\, \xi}(B)=\limsup_{l \rightarrow \infty}\mu_{n_l}(B)\\
&=\nu(B).
\end{split}
\end{equation}
Applying (\ref{MP005}), (\ref{ktime}) and (\ref{Bade}), we get
\begin{equation*}
 {\rm {BD}}^*\left({s^{-1}_{t_1}}F\cap {s^{-1}_{t_2}}F\cap...\cap {s^{-1}_{t_k}}F\right)\geq {\rm {BD}}^*(F)^k-\epsilon.
\end{equation*}
Hence the proposition is obtained.
\end{proof}

\begin{lemma}\label{BD}
Let $G$ be a countably infinite amenable group, $S$ be a subset of $G$ with ${\rm {BD}}^{*}(S)>0$ and $W\subseteq G$ be an infinite set (i.e. $|W|=\infty$). Then there are two distinct elements $l_{1}, l_{2}\in W$ such that
$${\rm {BD}}^{*}(l_{1}^{-1}S\cap l_{2}^{-1}S)\geq \frac{{\rm {BD}}^{*}(S)^{2}}{2}.$$
\end{lemma}

\begin{proof}
Set $W=\{s_1, s_2, \cdots ,s_n, \cdots\}$ ($s_i\not=s_j$ for $i\not=j$).  By Proposition \ref{mea.}, taking
\begin{equation*}
 k=2, \epsilon={\rm {BD}}^{*}(S)^{2}/2, \quad {\text {and }}\quad  N=N({\rm {BD}}^{*}(S), 2, {\rm {BD}}^{*}(S)^{2}/2)
\end{equation*}
as in Proposition \ref{mea.}, for the tuple $\{s_1, s_2, \cdots , s_n\}\subseteq W$ with $n\geq N$ there exist $1\leq t_{1}<t_{2}\leq n$ such that
$${\rm {BD}}^{*}(s_{t_{1}}^{-1}S\cap s_{t_{2}}^{-1}S)\geq\frac{1}{2}{\rm {BD}}^{*}(S)^{2}>0.$$

Let $l_1=s_{t_{1}}$ and $l_2=s_{t_{2}}$. Then the proof is completed.
\end{proof}

The main result in this section is:
\begin{theorem}\label{hto.}
Let $G$ be a countably infinite abelian group and $X$ be a compact metric space without isolated points.
Suppose the action $G\curvearrowright X$ is transitive. If $h_{{\rm {top}}}(X, G)>0$, then the action $G\curvearrowright X$ is Weyl-mean sensitive.
\end{theorem}

\begin{proof}
Since the group $G$ is abelian, the three concepts of Banach-, Weyl-mean and Besicovitch-$\mathcal{F}$-mean equicontinuous are equivalent. So
it suffices to prove there exists a transitive point $x_{0}$ which is not a Weyl-mean equicontinuous point by Proposition \ref{Prop7} and Theorem \ref{relation}.

As $G$ is abelian, thus the group $G$ is an amenable group.
Since $h_{{\rm {top}}}(X, G)>0$ and Theorem \ref{d02} and Definition \ref{b}, there exists IE pair $(x_{1},x_{2})\in IE_{2}(X,G)\setminus \bigtriangleup_{2}(X)$ satisfying for any nonempty open neighborhood $V_{1}\times V_{2}\ni (x_{1},x_{2})$, $\mathbf{A}'\doteq (V_{1}, V_{2})$ has positive independence density, i.e.,
$$
I(\mathbf{A}')=\inf_{F}\frac{\varphi_{\mathbf{A}'}(F)}{|F|}>0,
$$
where $\varphi_{\mathbf{A}'}(F)=\max\{|F\cap J|:J\quad\text{is an independence set for}\   \mathbf{A}'\}$ and $F$ ranges over all nonempty finite subset of $G$. Since $x_1\not=x_2$, we choose two open sets $U_{i}\, (i=1, 2)$ which are neighborhood of $x_{i}$ with
\begin{equation}\label{delta0}
  d(U_{1},U_{2})>2\delta_{0}>0 \quad {\text {and}}\quad I(\mathbf{A})=\inf_{F}\frac{\varphi_{\mathbf{A}}(F)}{|F|}>0.
\end{equation}
Here $\mathbf{A}=(U_{1}, U_{2})$. Thus, by Proposition \ref{c} and the equation \ref{ppt}, there exists an independence set $J$ for $\mathbf{A}$ such that
$$
{\rm {BD}}^{*}(J)\geq I(\mathbf{A})>0.
$$

Since $G\curvearrowright X$ is transitive and $X$ has no isolated points, by Propositions \ref{d} and \ref{e}, we know that the set Tran$(X,G)$ of points in $X$ with dense orbit and  the set Re$(X,G)$ of recurrent points are both dense $G_{\delta}$ sets of $X$. Applying Baire category theorem we have Tran$(X,G)\cap {\rm {Re}}(X,G)$ is also dense $G_{\delta}$ set of $X$, which means ${\rm {Tran}}(X,G)\cap {\rm {Re}}(X,G)\neq\emptyset$.

Let $x_{0}\in {\rm {Tran}}(X,G)\cap {\rm {Re}}(X,G)$. In what follows, we will show that $x_{0}$ is not a Weyl-mean equicontinuous point.

For each $\delta>0$, denote by
$$
G(x,B(x,\delta))=\{s\in G: sx\in B(x,\delta)\}.
$$
The cardinality of the set $G(x_{0},B(x_{0},\delta))$ is infinite because $x_{0}$ is a recurrent point.

Take $m_0\in \mathbb{N}$ satisfying
\begin{equation}\label{sss}
  \frac{1}{m_0}<\frac{\delta_0}{7}{\rm {BD}}^*(J)^2.
\end{equation}
Here $\delta_0$ is defined in \ref{delta0}. Recall that, from (\ref{xxx}),
\begin{equation*}
 \mathcal{E}_{1/m_0}=\bigg\{ x\in X : \exists \, \delta>0, \, \forall\, y, z\in B(x, \delta), \, D(y, z)<1/m_0 \bigg\}.
\end{equation*}
The rest of the proof we will prove the following assertion.

\vspace{0.1 cm}
\noindent {\bf {Claim. }}\   $x_0\not\in \mathcal{E}_{1/m_0}$.
\vspace{0.1 cm}

Suppose that
\begin{equation}\label{family01}
  x_0\in \mathcal{E}_{\frac{1}{m_0}}.
\end{equation}
Then there exists $\delta^*>0$ which depends on $x_0$ and $m_0$ such that
\begin{equation}\label{family02}
  D(y, z)<\frac{1}{m_0} \quad {\text {for all}} \quad  y, z\in B(x, \delta^*).
\end{equation}

Let
\begin{equation}\label{family03}
 0<\delta'<\min \{\delta_0,\   \delta^*\}.
\end{equation}

Since $x_{0}$ is a recurrent point, the cardinality of the set $G(x_{0},B(x_{0},\delta'))$  is infinite.
Recall that BD$^*(J)>0$ for the independent set $J$ for $\mathbf{A}$.
It follows from Lemma \ref{BD} that there are two distinct elements
\begin{equation}\label{vvv}
  l_{1}, l_{2}\in G(x_{0},B(x_{0},\delta'))
\end{equation}
such that
$$
{\rm {BD}}^{*}(l^{-1}_{1}J\cap l^{-1}_{2}J)\geq\frac{1}{2}{\rm {BD}}^{*}(J)^{2}.
$$

For each $g\in l^{-1}_{1}J\cap l^{-1}_{2}J$, there exist $g_{1},g_{2}\in J$ such that $g=l^{-1}_{1}g_{1}=l^{-1}_{2}g_{2}$. So we define a map
\begin{equation*}
 \varphi : l_1^{-1}J\cap l_2^{-1}J\rightarrow 2^J \quad \text {by}\quad \varphi(g)=\{g_1,g_2\}=\{l_{1}g,l_{2}g\}.
\end{equation*}
Given $s\in l^{-1}_{1}J\cap l^{-1}_{2}J $.  It is clear that the number of elements $g\in l^{-1}_{1}J\cap l^{-1}_{2}J$ such that $\varphi(g)\cap \phi(s)\neq \emptyset$ is at most three. Indeed, $g$ can only take $s$, $l_1^{-1} l_2 s$, or $l_2^{-1} l_1 s$.

Let $H$ be a maximal subset of $l^{-1}_{1}J\cap l^{-1}_{2}J$ with the property which is, for every pair $g,s\in H $ and $g\neq s $, $\varphi(g)\cap\varphi(s) =\emptyset$ (Zorn's Lemma guarantees the existence of the set $H$). Now we claim that
\begin{equation*}
 l^{-1}_{1}J\cap l^{-1}_{2}J \subseteq H\cup  l_1^{-1} l_2H\cup l_2^{-1} l_1H.
\end{equation*}
Otherwise, there exists an element $g^*\in l^{-1}_{1}J\cap l^{-1}_{2}J$ but $g^*\not\in H\cup  l_1^{-1} l_2H\cup l_2^{-1} l_1H$. So
\begin{equation}\label{zhub001}
 \varphi(g^*)\cap \varphi(h)=\emptyset \quad {\text {for all }}\quad h\in H.
\end{equation}
Indeed, if $\varphi(g^*)\cap \varphi(h_0)\not=\emptyset$ for some $h_0\in H$, then, by above argument, we know that $g^*=h_0$, $g^*=l_1^{-1} l_2 h_0$, or
$g^*=l_2^{-1} l_1 h_0$ which contradicts that $g^*\not\in H\cup  l_1^{-1} l_2H\cup l_2^{-1} l_1H$.
Hence, by (\ref{zhub001}), we deduce that the set $H\cup \{g^*\}$ satisfies the property that, for every pair $g, s\in H\cup \{g^*\}$ and $g\neq s $,  $\varphi(g)\cap\varphi(s)=\emptyset$. Noting $g^*\not\in H$ and $g^*\in l^{-1}_{1}J\cap l^{-1}_{2}J $, this contradicts that the set $H$ be a maximal subset $l^{-1}_{1}J\cap l^{-1}_{2}J$ with such property.

Hence, we get
\begin{equation}\label{zhub002}
 l^{-1}_{1}J\cap l^{-1}_{2}J \subseteq H\cup  l_1^{-1} l_2H\cup l_2^{-1} l_1H.
\end{equation}
According to Property \ref{right} and $G$ bing abelian, one has
\begin{equation*}
  {\rm{BD}}^*(H)={\rm{BD}}^*(l_1^{-1} l_2H)={\rm{BD}}^*(l_2^{-1} l_1H).
\end{equation*}
Combing with (\ref{zhub002}), it follows that
\begin{equation*}
 {\rm{BD}}^*(l^{-1}_{1}J\cap l^{-1}_{2}J)\leq {\rm{BD}}^*(H)+{\rm{BD}}^*(l_1^{-1} l_2H)+{\rm{BD}}^*(l_2^{-1} l_1H)=3{\rm{BD}}^*(H).
\end{equation*}
Therefore, we have
\begin{equation}\label{zhub003}
  {\rm {BD}}^{*}(H)\geq\frac{1}{3}{\rm {BD}}^{*}(l^{-1}_{1}J\cap l^{-1}_{2}J)\geq\frac{1}{6}{\rm {BD}}^{*}(J)^{2}.
\end{equation}

Recall that $G$ is an amenable group as $G$ is abelian and Theorem \ref{abiean}.
By (\ref{ppt}), we know that
\begin{equation}\label{zhub004}
  {\rm {BD}}^{*}(E)=\sup_{\mathcal{F}}\limsup_{n\rightarrow\infty}\frac{|E\cap F_{n}|}{|F_{n}|},
\end{equation}
where the supremum is taken over all F{\o}lner sequences $\mathcal{F}=\{F_{n}\}_{n\in \mathbb{N}}$ of $G$.
So, by (\ref{zhub003}) and (\ref{zhub004}), there is a F{\o}lner sequence $\{F_{n}\}$ of $G$ satisfying
$$
\limsup_{n\rightarrow\infty}\frac{|H\cap F_{n}|}{|F_{n}|}>\frac{1}{7}{\rm {BD}}^{*}(J)^{2}.
$$
Therefore, there exists a subsequence $\{m_n\}_{n=1}^{\infty}$ of $\mathbb{N}$ such that $m_n<m_{n+1}$, $m_n\geq n$ and
\begin{equation}\label{zhub005}
  \frac{|F_{m_n}\cap H|}{|F_{m_n}|}>\frac{1}{7}{\rm {BD}}^{*}(J)^{2}.
\end{equation}

We denote by $J_1=l_1 H$ and $J_2=l_2 H$. Since $H\subseteq l^{-1}_{1}J\cap l^{-1}_{2}J$, we immediately have $J_1\cup J_2 \subseteq J$.
Furthermore, we have $J_1 \cap J_2=\emptyset$. Indeed, if $J_1 \cap J_2\not=\emptyset$, then there are $h_1, h_2\in H$ such that
$l_1 h_1=l_2 h_2$. As $l_1\not=l_2$, it follows that $h_1\not=h_2$. Note that
\begin{equation*}
  \varphi(h_1)=\{l_1 h_1, l_2 h_1\}=\{l_2 h_2, l_2 h_1\} \quad {\text {and}}\quad \varphi(h_2)=\{l_1 h_2, l_2 h_2\}.
\end{equation*}
Thus $\varphi(h_1)\cap \varphi(h_2)\not=\emptyset$ and $h_1\not=h_2 \in H$ which contradicts the definition of $H$. Hence, $J_1 \cap J_2\not=\emptyset$.

Let $n\in \mathbb{N}$. The inequality (\ref{zhub005}) implies that $F_{m_n}\cap H\not=\emptyset$. Denote by
\begin{equation*}
  T_n=F_{m_n}\cap H.
\end{equation*}
Then we define the maps $\psi_i : T_n\rightarrow J_i (i= 1, 2)$ as follows:
\begin{equation*}
 \psi_i(s):=l_i s \quad {\text {for}}\quad s\in T_n.
\end{equation*}
It is easy to see that
\begin{equation*}
 \psi_i(T_n)\subseteq J_i\subseteq J\   (\text {for}\  i=1, 2) \quad {\text {and}}\quad \psi_1(T_n)\cap \psi_2(T_n)=\emptyset \quad {\text {as}}\quad J_1\cap J_2=\emptyset.
\end{equation*}
From the definition of the independence set of $J$ for $\mathbf{A}=(U_1, U_2)$ (see Definition \ref{indef01} and Definition \ref{indef02}), we get that
\begin{equation*}
 \left(\bigcap_{s\in T_n}(\psi_1(s))^{-1} U_1\right)\cap\left(\bigcap_{s\in T_n}(\psi_2(s))^{-1} U_2\right)\neq \emptyset,
\end{equation*}
that is,
\begin{equation*}
 \left(\bigcap_{s\in T_n}(\psi_1(s))^{-1} U_1\right)\cap\left(\bigcap_{s\in T_n}(\psi_2(s))^{-1} U_2\right)\quad {\text {is an nonempty open set of }} \  X.
\end{equation*}
Recall that $X=\overline{G x_{0}}$. Hence, we get
\begin{equation*}
 G x_{0} \cap \left(\bigcap_{s\in T_n}(\psi_1(s))^{-1} U_1\right)\cap\left(\bigcap_{s\in T_n}(\psi_2(s))^{-1} U_2\right)\neq \emptyset.
\end{equation*}
Choose a point $y_n\in X$ with
\begin{equation*}
 y_n\in G x_{0} \cap \left(\bigcap_{s\in T_n}(\psi_1(s))^{-1} U_1\right)\cap\left(\bigcap_{s\in T_n}(\psi_2(s))^{-1} U_2\right).
\end{equation*}
So
\begin{equation}\label{zz00}
  y_n=g_n x_0\quad {\text {for some  }} \  g_n\in G.
\end{equation}
Moreover, for each $g\in T_n$, since $G$ is abelian one has
\begin{equation}\label{zhub006}
  g (l_1 y_n)=(l_1 g)\, y_n=\psi_1(g) y_n\in U_1\quad {\text {and}}\quad g (l_2 y_n)=l_2 g\, y_n=\psi_2(g) y_n\in U_2.
\end{equation}
Combining with \ref{zz00} and $G$ being an abelian group, we get
\begin{equation*}
  (g_n g)(l_1 x_0)=(g l_1) y_n\in U_1\quad {\text {and}}\quad (g_n g)(l_2 x_0)=(g l_2) y_n\in U_2.
\end{equation*}
Therefore, we obtain that, for each $g\in T_n$,
\begin{equation*}
  g_n g\in G(l_{1} x_0, U_{1})\cap G(l_{2} x_0, U_{2}),
\end{equation*}
that is,
\begin{equation}\label{zhub007}
  g_n T_n=g_n (F_{m_n}\cap H)\subseteq G(l_{1} x_0, U_{1})\cap G(l_{2} x_0, U_{2}).
\end{equation}
Recall that $d(U_{1},U_{2})>2\delta_{0}$. Hence, one has
\begin{equation}\label{zz01}
  d(s l_1 x_0, s l_2 x_0)>2\delta_{0} \quad {\text {for every  }}\   s\in g_n T_n=g_n (F_{m_n}\cap H).
\end{equation}
Therefore, we have
\begin{equation}\label{ttt}
\begin{split}
\limsup_{n\rightarrow\infty} \frac{1}{|g_n F_{m_n}|} \sum_{s\in g_n F_{m_n}} d\left(s (l_1 x_0), s (l_2 x_0)\right)&
\geq \limsup_{n\rightarrow\infty} \frac{1}{|g_n F_{m_n}|} \sum_{s \in g_n (F_{m_n}\cap H)} d\left(s (l_1 x_0), s (l_2 x_0)\right)\\
& \geq \limsup_{n\rightarrow\infty} \frac{1}{|g_n F_{m_n}|}\cdot 2 \delta_0\cdot |g_n (F_{m_n}\cap H)|\\
& = \limsup_{n\rightarrow\infty} 2 \delta_0 \frac{|F_{m_n}\cap H|}{|F_{m_n}|}  \quad {\text {(by \ref{zhub005})}}\\
& > \frac{\delta_0}{7}{\rm {BD}}^{*}(J)^{2}.
\end{split}
\end{equation}
Denote by
\begin{equation*}
  \mathcal{F}'=\{g_n F_{m_n}\}_{n\in\mathbb{N}}.
\end{equation*}
Since $\{F_n\}_{n\in\mathbb{N}}$ is a F{\o}lner sequence of the abelian group $G$, $\mathcal{F}'$ is also a F{\o}lner sequence of $G$.
The inequality (\ref{ttt}) shows that
\begin{equation}\label{iii}
  D(l_1 x_0, l_2 x_0)\geq D_{\mathcal{F}'} (l_1 x_0, l_2 x_0)>\frac{\delta_0}{7}{\rm {BD}}^{*}(J)^{2}. \quad
  {\text {(see Definition \ref{mmm})}}
\end{equation}
Meanwhile, the above inequality implies that $l_1 x_0\not=l_2 x_0$.

Recall that, from (\ref{vvv}) and (\ref{family03}),
\begin{equation*}
  l_{1}, l_{2}\in G(x_{0},B(x_{0},\delta'))\quad {\text {and}}\quad \delta'<\delta^*,
\end{equation*}
that is,
\begin{equation}\label{bbb}
 l_1 x_0, \,l_2 x_0\in B(x_{0}, \delta^*).
\end{equation}

Combing with (\ref{iii}), (\ref{bbb}) and (\ref{sss}), one has
\begin{equation*}
 l_1 x_0, \,l_2 x_0\in B(x_{0}, \delta^*) \quad {\text {and}}\quad D(l_1 x_0, l_2 x_0)>\frac{1}{m_0}.
\end{equation*}
This contradicts the inequality (\ref{family02}). Hence we obtain that
\begin{equation}\label{family004}
  x_0\not\in\mathcal{E}_{\frac{1}{m_0}}.
\end{equation}

Recall that $\mathcal{E}$ denote the set of all Weyl-mean equicontinuous points of the dynamic system $(X, G)$.
From Proposition \ref{g} we know that
\begin{equation}\label{ccc}
 \mathcal{E}=\bigcap_{m=1}^{\infty}\mathcal{E}_{1/m}.
\end{equation}

By (\ref{family004}) and (\ref{ccc}), we get
\begin{equation*}
  x_0\not\in \mathcal{E}.
\end{equation*}
Therefore, $x_0$ is not a Weyl-mean equicontinuous point of $(X, G)$. So, $x_0$ is a Weyl-mean sensitive point of $(X, G)$.
By the assumption $G\curvearrowright X$ being transitive and Proposition \ref{Prop7}, we deduce that
$G\curvearrowright X$ is Weyl-mean sensitive.

Hence, the theorem is proved.
\end{proof}

\noindent {\bf {Proof of Theorem 1.1.}}

It follows from Theorem \ref{hto.}, Theorem \ref{relation} and Theorem \ref{abiean}.

\section{an application}

In order to get our result, we need to prepare for the following concepts and theorems.

\begin{definition}(\cite{KL})
By a p.m.p. (probability-measure-preserving) action of $G$, we mean
an action of $G$ on a standard probability space $(X, \mu)$ by measure-preserving transformations.
In this case, we will combine together the notation and simply write
$G\curvearrowright(X, \mu)$.
\end{definition}

Given an action $G\curvearrowright X$ on a compact metric space $X$, we say that a set $A\subseteq X$ is $G$-{\it {invariant}} if
$G A = A$, which is equivalent to $GA\subseteq A$. When the action is probability-measure preserving
and $A$ is a measurable set, we interpret $G$-invariance to mean $GA=A$
modulo a null set, i.e., $\mu(sA\triangle A)= 0$ for all $s\in G$.

\begin{definition}(\cite{KL})
The action $G\curvearrowright (X,\mu)$ is said to be {\it {ergodic}} if $\mu(A)=0$ or 1 for every $G$-invariant measurable set $A\subseteq X$.
\end{definition}

Any dynamical system with an amenable group action admits invariant probability measures and the
ergodic measures can be shown to be the extremal points of the set of invariant
probability measures (see, for example, the monographs \cite{DGS,PW1}). Let $\mathcal{M}(X)$, $\mathcal{M}_{G}(X)$ and $\mathcal{M}_{G}^{e}(X)$ denote the sets of all Borel probability measures on $X$, $G$-invariant regular Borel probability measures  on $X$ and ergodic measures in $\mathcal{M}_{G}(X)$, respectively.

\begin{prop}(\cite[\textit{Proposition 2.5}]{KL})\label{jjj}
For a p.m.p. action $G\curvearrowright(X, \mu)$ the following are equivalent:
\begin{enumerate}
\item  the action is ergodic,
\item $\mu(A)=0$ or 1 for every measurable set $A\subseteq X$ satisfying $sA = A$ for all $s\in G$ (i.e., G-invariance in the strict sense),
\item for all sets $A, B\subseteq X$ of positive measure there is an $s \in G$ such that $\mu(sA\cap B)>0$.
\end{enumerate}
\end{prop}

Now, we recall the conceptions of amenable measure entropy (see \cite{HLZ} and \cite{KL}).

Let $G$ be a amenable group and $G\curvearrowright(X, \mu)$ be a p.m.p. action. Let $$\mathscr{P}=\{A_{1},A_{2},\cdots,A_{n}\}$$ be a finite partition of $X$ and $F$ be a nonempty finite subset of $G$. Set $\mathscr{P}^{F}$ for the join $\bigvee_{s\in F} s^{-1}\mathscr{P}$, $$h(\mathscr{P})=\inf_{F}\frac{1}{|F|}H(\mathscr{P}^{F})$$
where $F$ ranges over nonempty finite subsets of $G$ and $$H(\mathscr{P})=\sum_{i=1}^{n}-\mu(A_{i})\log\mu(A_{i}).$$
The entropy of the action $G\curvearrowright(X, \mu)$ is
$$h_{\mu}(X, G) = \sup_{\mathscr{P}} h(\mathscr{P})$$
where $\mathscr{P}$ ranges over all finite partitions of $X$.

The {\it {support}} of a measure $\mu\in \mathcal{M}(X)$, denoted by supp$(\mu)$, is the smallest closed subset $C$ of $X$ such that $\mu(C)=1$ (see \cite{KS}), that is,
\begin{equation*}
 {\rm {supp}}(\mu)= \bigcap_{\substack{K \text { is closed}\\
\text {and }\mu(K)=1}} K.
\end{equation*}

 The following fact is well known.

\begin{fact}\label{supp0001}
We have
\begin{equation*}
\begin{split}
{\rm {supp}}(\mu)&=\{x\in X: {\text {for every open neighborhood $U$ of $x$}} , \mu(U)>0\} \\
&=X\backslash \bigcup\{U\subset X:U {\text { is open and }} \mu(U)=0\}.
\end{split}
\end{equation*}
\end{fact}

Topological entropy is related to measure entropy by the variational principle
which asserts that for a continuous map on a compact metric space the topological entropy equals the supremum of the
measure entropy taken over all the invariant probability measures. The following conclusion is the variational principle of the version of the amenable group action that we need in this paper.

\begin{theorem}\label{VPTE0001}(\cite[\textit{Theorem 5.2}]{HYZ})
(Variational principle of topological entropy)Let $G$ be an amenable group and $X$ be a compact metric space. Then
$$h_{{\rm {top}}}(G, X)=\sup_{\mu\in \mathcal{M}_{G}(X)}h_{\mu}(G,X)=\sup_{\mu\in \mathcal{M}_{G}^{e}(X)}h_{\mu}(G,X).$$
\end{theorem}

As an application of our main result, we have the following result.

\noindent {\bf {Theorem 1.2}} {\it {
Let $G$ be a countable abelian group, $X$ be a compact metric space and $G\curvearrowright X$ be a continuous action. If $G\curvearrowright X$ is Banach-mean equicontinuous, then
\begin{equation*}
 h_{{\rm {top}}}(X, G)=0.
\end{equation*}
}}
\begin{proof}
Let $\mu$ be an ergodic invariant measure on the action $G\curvearrowright X$. Denote by $X_0={\rm {supp}}(\mu)$ the support of the ergodic invariant measure $\mu$. It is clear that $X_0$ is a $G$-invariant closed subset of $X$ and $G\curvearrowright(X_0, \mu)$ is also ergodic. Moreover, we have $h_{\mu}(X, G)=h_{\mu}(X_0, G)$.

In what follows we show that $h_{\mu}(X_0, G)=0$.

Let $U$, $V$ be any pair nonempty open sets of $X_0$, then $\mu(U)\mu(V)>0$ since supp$(\mu)=X_0$ and Fact \ref{supp0001}. Thus there is an element $s\in G$ such that $\mu(U\cap sV)>0$ by $G\curvearrowright(X_0, \mu)$ being ergodic and Proposition \ref{jjj}, that is,
\begin{equation*}
 U\cap sV\not=\emptyset.
\end{equation*}
Note $X_0$ be a compact metric space.
Hence the action $G\curvearrowright X_0$ is topological transitive.

Now we divide two cases to complete our proof.

\noindent{\it Case 1. }\  $X_0$ has no isolated points.

Since $G\curvearrowright X$ is Banach-mean equicontinuous, it is clear that $G\curvearrowright X_0$ is also Banach-mean equicontinuous.
By Theorem \ref{maintheorem1} and $X_0$ having no isolated points, we get
\begin{equation*}
  h_{{\rm {top}}} (X_0, G)=0.
\end{equation*}
Note that $\mu|_{X_0}$ is an ergodic measure of $G\curvearrowright X_0$. Then by Theorem \ref{VPTE0001}, we obtain
\begin{equation*}
 h_{\mu}(X_0, G)=0.
\end{equation*}
Therefore, we have
\begin{equation*}
  h_{\mu}(X, G)=h_{\mu}(X_0, G)=0.
\end{equation*}
Recall that $\mu$ be any ergodic invariant measure on the action $G\curvearrowright X$. Again applying Theorem \ref{VPTE0001}, we deduce that
\begin{equation*}
  h_{{\rm {top}}} (X, G)=0.
\end{equation*}

\noindent{\it Case 2. }\  $X_0$ has isolated points.

Suppose that $x_0\in X_0$ is an isolated point of $X_0$. So the single point set $\{x_0\}$ is an open set of $X_0$. Let $V\subseteq X_0$ be any open set. Since the action $G\curvearrowright X_0$ is topological transitive, there is $s\in G$ such that $s x_0\in V$. This fact implies the orbit of $x_0$ is dense in $X_0$, that is, $\overline{G x_0}=X_0$.

Note that $x_0\in {\rm {supp}}(\mu)$ and the single point set $\{x_0\}$ is an open set. Thus one has $\mu(\{x_0\})>0$. Since $\mu(X_0)=1$ we deduce that
the cardinality of the set $G x_0$ is finite (i.e. $|G x_0|<\infty$). Combining with $\overline{G x_0}=X_0$, we get that the cardinality of the space $X_0$ is finite (i.e. $|X_0|<\infty$). By the definition of topological entropy, it is easy to see that
\begin{equation*}
  h_{{\rm {top}}} (X_0, G)=0.
\end{equation*}
In the following, with an similar argument as in Case 1, we can obtain
\begin{equation*}
  h_{{\rm {top}}} (X, G)=0.
\end{equation*}

Hence the theorem is proved.

\end{proof}


\bigskip \noindent{\bf Acknowledgement}.
 The authors are very grateful Prof. Hanfeng Li for his generous sharing of his knowledge
 and his ideas about the topic with him. The authors would also like to thank Prof. Jian Li for his
 helpful suggestions.


\bibliographystyle{amsplain}

\begin{thebibliography}{10}

\bibitem {A1} J. Auslander,
\textit{Mean-L-stable systems.}
Illinois J. Math. \textbf{3} (1959), 566--579.

\bibitem {A2} J. Auslander,
\textit{Minimal Flows and Their Extensions.}
North-holland, 1988.


\bibitem {AAB} E. Akin, J. Auslander and K. Berg.
\textit{when is a transitive map chaotic?}
Convergence in Ergodic Theory and Probability(Conlumbus,OH,199
3)(Ohio State university Mathematic Research Institute Publication,5).de Gruyter, Berlin,1996,pp.25-40.


\bibitem {BBF} Mathias Beiglb\"{o}ck, Vitaly Bergelson, and Alexander Fish,
\textit{Sumset phenomenon in countable amenable groups.}
Adv. Math. \textbf{223} (2010),  416-432.


\bibitem {BLM} M. Baake, D. Lenz, and R. V. Moody,
\textit{Characterization of model sets by dynamical systems.}
Ergos. Th. \& Dynam. Sys. \textbf{27}(2) (2007),  341-382.


\bibitem {CC} T. Ceccherini-Silberstein, M. Coornaert,
\textit{Cellular Automata and Groups.}
Springer-Verlag, 2010.


\bibitem {DGS} M. Denker, C. Grillenberger and K. Sigmund,
\textit{Ergodic Theory on Compact Spaces.}
Lecture Notes in Mathematics, vol. 527, Springer, Berlin,  1976.


\bibitem {DHZ1} T. Downarowicz, D. Huczek, and G. Zhang,
\textit{Tilings of amenable groups.}
J. Reine Angew. Math.  \textbf{747}  (2019), 277-298.



\bibitem {F} S. Fomin,
\textit{On dynamical systems with pure point spectrum.}
Dokl. Akad. Nauk SSSR \textbf{77}(4) (1951),  29-32.


\bibitem {F1} H. Furstenberg,
\textit{Recurrence in Ergodic Theory and Combinatorial Number Theory.}
M.B. Porter Lectures, Princeton University Press, Princeton, NJ, 1981.


\bibitem {FGL} G. Fuhrmann, M. Gr\"{o}ger, and D. Lenz,
\textit{The structure of mean equacontinuous group actions.} arXiv preprint. Available at: https://arxiv.org/pdf/1812.10219v1.

\bibitem {GR} F. Garc\'{\i}a-Ramos,
\textit{Weak forms of topological and measure theoretical equicontinuity: relationships with discrete spectrum and sequence entropy.}
Ergos. Th. \& Dynam. Sys. 37(4) (2017), 1211-1237.


\bibitem {GRM} F. Garc\'{\i}a-Ramos and B. Marcus,
\textit{Mean sensitive, mean equicontinuous and almost periodic functions for dynamical systems.}
Discrete Contin. Dyn. Syst. 39 (2019), no.2, 729-746.


\bibitem {GM} S. Glasner, D. Maon.
\textit{Rigidity in topological dynamicals}
Ergos. Th.\& Dynam. Sys. 9 (1989), 309-320.


\bibitem {GW} E. Glasner and B. Weiss.
\textit{Sensitive dependence on initial conditions.}
Nonlinearity 6(6)  (1993), 1067-1075.


\bibitem {HLZ} X. Huang, J. Liu,  C. Zhu,
\textit{The Bowen topological entropy of subsets for amenable group actions.}
J. Math. Anal. Appl. \textbf{472} (2) (2019), 1678--1715.


\bibitem {LTY1} J. Li, S. Tu, X. Ye,
\textit{Mean equicontinuity and mean sensitivity.}
Ergod. Th. Dynam. Sys. \textbf{35} (8) (2015), 2587-2612.

\bibitem {HLTXY} W. Huang, J. Li, J. Thouvenot, L. Xu and X. Ye,
\textit{Mean equicontinuity, bounded complexity and discrete
spectrum.} arXiv preprint. Available at: https://arxiv.org/pdf/1806.02980.


\bibitem {HYZ} W. Huang, X. Ye, G. Zhang,
\textit{Local entropy theory for a countable discrete amenable group action.}
Journal of Functional Analysis, \textbf{261} (4) (2011), 1028-1082.


\bibitem {KL} D. Kerr, H. Li,
\textit{Ergodic Theory:Independence and Dichotomies.}
Springer, 2016.


\bibitem {P} K.R. Parthasarathy,
\textit{Probability measures on metric spaces.}
Academic Press Inc., New York, 1967.


\bibitem {P1}  M. \L\c{a}cka, M. Pietrzyk,
\textit{Quasi-uniform convergence in dynamical systems generated by an amenable group action.} J. Lond. Math. Soc. in press, 2018.


\bibitem {KS} K. Sigmund,
\textit{On minimal centers of attraction and generic points.}
J. Reine Angew. Math. \textbf{295} (1977), 72--79.


\bibitem {PW1} P. Walters,
\textit{An Introduction to Ergodic Theory.}
Springer, New York, 1982.

\bibitem {S} B. Scarpellini,
\textit{Stability properties of flows with pure point spectrum.}
J. London Math. Soc. \textbf{26} (2) (1982), no.3, 451-464.


\end{thebibliography}

\end{document}